\documentclass[A4, 11pt]{article}
\usepackage{latexsym}
\usepackage{amsmath}
\usepackage{amsthm}
\usepackage{amssymb}
\usepackage{amsfonts}
\usepackage{esint}
\usepackage[all]{xy}
\usepackage{graphics,graphicx}
\usepackage{accents}
\usepackage[french]{babel}
\usepackage[utf8]{inputenc}
\usepackage[T1]{fontenc}
\usepackage{color}
\usepackage{tikz-cd}
\usepackage{hyperref}
\theoremstyle{definition}
\newtheorem{defin}{D\'efinition}[section]
\theoremstyle{plain}
\newtheorem{theo}[defin]{Th\'eor\`eme}
\newtheorem{prop}[defin]{Proposition}

\newtheorem{lem}[defin]{Lemme}
\theoremstyle{definition}
\newtheorem{rem}[defin]{Remarque}

\newtheorem*{abstract2}{Abstract}

\def\C{{\mathbf C}}

\def\N{{\mathbf N}} 
\def\R{{\mathbf R}} 
\def\Q{{\mathbf Q}}
\def\Z{{\mathbf Z}}

\def\F{{\mathbf F}}

\def \mA { \mathcal{A}}

\def \mC {\mathcal{C}}
\def\Ker{ \operatorname{Ker} }
\def\Im{ \operatorname{Im} }
\def \Spec { \operatorname{Spec} }

\def \End {\operatorname{End} }
\def \Hom {\operatorname{Hom} }
\def \Tr {\operatorname{Tr} }
\def \Card {\operatorname{Card} }
\author{Séverin Philip}
\begin{document}
\title{Variétés abéliennes CM et grosse monodromie finie sauvage}

\maketitle

\begin{abstract2}
~In this paper we study the wild part of the finite monodromy groups of abelian varieties over number fields. We solve Grunwald problems for groups of the form $\Z/p\Z\wr \mathfrak{S}_n$ over number fields to build CM abelian varieties with maximal wild finite monodromy in the odd prime case. For the even prime case we prove a new bound on the 2-part of the order of the finite monodromy group for CM abelian varieties and build varieties that reach it.  

\end{abstract2}

\section{Introduction}

\subsection{Contexte et motivation} \label{chap2cont}

Les groupes de monodromie finie d'une variété abélienne sont d'abord introduits par Serre dans \cite{propgal} dans le cas des courbes elliptiques. Pour une variété abélienne $A$ de dimension quelconque sur un corps local à corps résiduel algébriquement clos, Grothendieck montre l'existence d'une plus petite extension, qui est galoisienne, sur laquelle $A$ atteint réduction semi-stable dans \cite{sga} exposé IX. Le groupe de monodromie finie de $A$ est alors défini comme le groupe de Galois de cette extension. Dans le cas où $A$ est une variété abélienne sur un corps de nombres $K,$ on obtient par cette construction pour chaque place non archimédienne $v$ de $K$ un groupe de monodromie finie en $v$ de $A,$ que l'on note $\Phi_{A,v},$ qui représente l'obstruction locale à la semi-stabilité de $A$. Ces groupes sont ensuite étudiés par Silverberg et Zarhin dans \cite{Sb1998} et \cite{Sb2004}. Dans \cite{Sb1998}, il est montré que le cardinal d'un groupe de monodromie finie $\Phi_{A,v}$ divise le ppcm des cardinaux des sous-groupes finis de $\mathrm{GL}_{2g}(\Q)$ où $g$ est la dimension de $A$. Ce ppcm, que l'on note $M(2g)$, est calculé par Minkowski en 1887 et est appelé borne de Minkowski. Pour tout nombre premier $p$ et tout entier naturel non nul $n$ on pose
$$r(n,p)= \sum\limits_{i\geq 0} \big \lfloor \frac{n}{p^i (p-1)} \big \rfloor.$$
Alors la borne de Minkowski est donnée par
$$M(n)=\prod\limits_p p^{r(n,p)}.$$

Cette borne de divisibilité est atteinte en dimension $1$ car le groupe $\mathrm{SL}_2(\F_3)$, qui est le groupe des automorphismes de la courbe supersingulière sur $\overline{\F_2}$, est le groupe de monodromie finie en $2$ de la courbe elliptique sur $\Q$
$$E\colon y^2=x^3-2x^2-x$$
(voir le paragraphe 5.9.1 \cite{propgal}). On a $\Card \mathrm{SL}_2(\F_3)= M(2)=24$. La liste des groupes de monodromie finie pour les surfaces abéliennes fait l'objet de \cite{Sb2004}. Plus précisément ils établissent une liste de groupes finis telle que toute surface abélienne sur un corps de nombres ou local prend ses groupes de monodromie finie dans la liste. Ensuite ils montrent que chacun des groupes de la liste est le groupe de monodromie finie d'une surface abélienne sur un corps local d'égale caractéristique. On peut vérifier que la borne de Minkowski n'est atteinte pour aucun des groupes de la liste, en revanche c'est le plus petit commun multiple des cardinaux des groupes listés. En particulier, on trouve pour chaque premier $p$ une surface abélienne $A_p$ sur un corps local $K_p$ (d'égale caractéristique $p$) telle que

$$v_p(\Card \Phi_{A_p})= v_p(M(4))=r(4,p).$$

On se propose dans cet article de montrer une assertion analogue pour les corps de nombres et en dimension arbitraire. On va donc s'intéresser à la construction de variétés abéliennes de toutes dimensions sur des corps de nombres dont la $p$-partie des groupes de monodromie finie est maximale. Cette étude nous amène à montrer une nouvelle borne pour la $2$-partie du cardinal d'un groupe de monodromie finie d'une variété abélienne potentiellement CM.

\subsection{\'Enoncé des résultats}
\begin{theo}\label{chap2main}
Soient $g$ un entier naturel non nul et $K$ un corps de nombres non ramifié en $2$. On note $p_1,\dots,p_n$ les diviseurs premiers impairs de $M(2g)$. Alors il existe une extension finie $L$ de $K$ telle que pour chaque $i\in \{1,\dots,n\}$ il existe une variété abélienne $A_i$ de dimension $g$ principalement polarisée sur $L$ et une place $v_i$ de $L$ avec
$$\Card \Phi_{A_i,v_i}= p_i^{r(2g,p_i)}$$
et il existe une variété abélienne principalement polarisée $A$ de dimension $g$ sur $L$ et une place $v$ de $L$ telles que 
$$\Card \Phi_{A,v}=2^{r(2g,2)+1-g}.$$
\end{theo}
La $p$-partie des groupes de monodromie finie dont le théorème assure l'existence est donc maximale, sauf lorsque $p$ est pair. La construction des variétés abéliennes $A_i$ repose sur la théorie de la multiplication complexe et le manque, lorsque $p=2,$ s'explique par le théorème suivant, résultat principal de la partie 5 qui avec le théorème \ref{Chap2theoprincipaldem} donne  l'égalité 
$$\Card \Phi_{A,v}=2^{r(2g,2)+1-g}$$
 dans l'énoncé précédent.
 
\begin{theo} \label{chap2majcmgen}
Soit $A$ une variété de dimension $g$ sur un corps de nombres $K$ telle que $A_{\overline{K}}$ est CM. Alors on a 
$$v_2(\Card \Phi_{A,v}) \leq r(2g,2)+1-g$$
pour toute place ultramétrique $v$ de $K$. De plus, l'égalité ne peut intervenir que lorsque une composante isotypique de $A_{\overline{K}}$ est isogène à une puissance de la courbe elliptique $y^2=x^3-x$. 
\end{theo}

Lorsque $A$ est isotypique et que $A$ a multiplication complexe par $Z,$ il est connu depuis \cite{serretate} que $\Phi_{A,v}$ est un sous-groupe du groupe des racines de l'unité de $Z$. Si l'on ne suppose plus que $A$ est CM mais seulement que $A_{\overline{K}}$ l'est alors le degré $[K_A:K]$ intervient dans la majoration où $K_A$ est le corps de définition des endomorphismes de $A_{\overline{K}}$. Ce degré est étudié dans \cite{r} et \cite{gk}; ces derniers obtiennent
$$v_2([K_A:K]) \leq r(2g,2)-g-1.$$
On obtient alors la borne en combinant les deux approches de façon adaptée à notre contexte.

La construction des variétés abéliennes de l'énoncé du théorème \ref{chap2main} se fait par une adaptation aux corps de nombres des techniques de \cite{Sb2004} qui sont utilisées dans le cas de corps locaux d'égales caractéristiques $p$. Précisément on construit ces variétés abéliennes comme formes tordues de variétés abéliennes avec multiplication complexe (CM).  Pour cela on effectue deux généralisations du théorème 4.3 de torsion de \cite{Sb2004}, l'une aux corps de nombres et l'autre avec des hypothèses réduites, qui permettent de construire des variétés abéliennes avec des groupes de monodromie finie prescrits. Pour déduire de ces théorèmes le résultat principal on a besoin d'une part de l'existence de variétés abéliennes avec des gros groupes d'automorphismes et d'autre part de l'existence d'extensions galoisiennes de corps de nombres avec des gros groupes d'inertie en des places choisies.

En partie 2 on montre l'existence de variétés abéliennes CM qui ont $\Z[\zeta_p]$ pour anneaux d'endomorphismes ainsi que les autres énoncés sur les variétés abéliennes CM et polarisations utiles à la suite du texte. L'obstruction à cette approche, pour $p=2,$ vient du fait que l'on ne dispose pas, en caractéristique nulle, d'une courbe supersingulière. On remarque que la courbe $E$ donnée en partie \ref{chap2cont} n'est pas CM.

La partie 3 est consacrée aux résultats d'arithmétique des corps de nombres nécessaires pour appliquer  les théorèmes de torsion. On commence pour cela de manière analogue à \cite{Sb2004} en établissant l'existence d'extensions locales pour des $p$-groupes prescrits. On résout ensuite des problèmes de Grunwald pour des produits en couronne de la forme $\Z/p^m\Z\wr \mathfrak{S}_n$. On traite à part la construction pour $p=2$ qui nécessite une attention particulière. 
Dans tous les cas, bien qu'on cherche à produire des groupes de monodromie finie qui sont des $p$-Sylow des groupes concernés, on travaille avec les groupes complets pour permettre la résolution du problème de Grunwald qui se pose (i.e. le passage d'extensions locales à une extension d'un corps de nombres). 

\subsection{Rappels sur les groupes de monodromie finie} \label{chap2rappel}

Soient $A$ une variété abélienne de dimension $g$ sur un corps de nombres $K$ et $v$ une place non archimédienne de $K$ de caractéristique résiduelle $p$. On note $I$ le groupe d'inertie d'une extension $\overline{v}$ de $v$ à $\overline{K}$. Soit $\ell\neq p$ un nombre premier. L'action du groupe de Galois absolu $\mathrm{Gal}(\overline{K}/K)$ sur le module de Tate $\operatorname{T}_{\ell} A$ de $A$ induit une représentation $\ell$-adique
$$\rho_{A,\ell} \colon \mathrm{Gal}(\overline{K}/K) \longrightarrow \mathrm{GL}_{2g}( \Q_{\ell}).$$ 

On note $G$ le groupe algébrique linéaire obtenu comme adhérence pour la topologie de Zariski de l'image de $\rho_{A,\ell}$ restreinte à $I$ dans $\mathrm{GL}_{2g}$. 
\begin{defin}
Le groupe de monodromie finie de $A$ en $v$ noté $\Phi_{A,v}$ est le groupe des composantes $G(\overline{K})/G^{\circ}(\overline{K})$.
\end{defin}
On note $I_{A,v}$ le noyau du morphisme naturel, surjectif $I\rightarrow G(\overline{K})/G^{\circ}(\overline{K})$. On a donc $\Phi_{A,v}=I/I_{A,v}$ par définition.

Il est démontré dans \cite{Sb1998} (théorème 4.2) que cette définition est équivalente à celle de Grothendieck et en particulier 
$$I_{A,v}=\{ \sigma\in I \mid \rho_{A,\ell}(\sigma) \text{ est unipotent}\}.$$
Dans le cas où $A$ a bonne réduction potentielle cette égalité devient
$$I_{A,v}=\Ker \rho_{A,\ell}.$$

 Le résultat fondamental sur les groupes de monodromie finie est le suivant.

\begin{theo}(Grothendieck)
La variété abélienne $A$ a réduction semi-stable en $v$ si et seulement si $\Phi_{A,v}=\{1\}$. 
\end{theo}

Il suit de la définition et du résultat fondamental les propriétés suivantes :

\medskip
\begin{itemize}
\item[(i)] Le sous-groupe $I_{A,v}$ de $I$ définit une extension galoisienne $(K_v^{\mathrm{nr}})_{A,s}$ de $K_v^{\mathrm{nr}}$ qui est la plus petite extension sur laquelle $A_{K_v^{\mathrm{nr}}}$ atteint réduction semi-stable. Autrement dit, si $L$ est une extension de $K_v^{\mathrm{nr}}$ telle que $A_L$ a réduction semi-stable alors $(K_v^{\mathrm{nr}})_{A,s} \subset L$. En particulier on a 
$$\Card \Phi_{A,v}= [(K_v^{\mathrm{nr}})_{A,s} : K_v^{\mathrm{nr}}].$$

\item[(ii)] Les groupes de monodromie finie sont invariants par isogénie et puissance. De plus, si $B$ est une variété abélienne sur $K$ on a l'inclusion
$$(K_v^{\mathrm{nr}})_{A\times B,s} \subset (K_v^{\mathrm{nr}})_{A,s}(K_v^{\mathrm{nr}})_{B,s}.$$

\item[(iii)] Si $L$ est une extension de $K$ et $w\mid v$ est une place de $L$ non ramifiée alors
$$\Phi_{A_L,w}= \Phi_{A,v}.$$
\end{itemize}

On finit par remarquer que suite au résultat fondamental et à la propriété (i) on peut toujours, par un résultat d'approximation faible (voir par exemple la proposition \ref{locglob} ci-dessous), construire une extension finie $L$ de $K$ ramifiée seulement en $v$ et des places de caractéristiques résiduelles arbitrairement grandes telle que $A_L$ a réduction semi-stable aux places de $L$ au-dessus de $v$.

\section{Existence de variétés CM principalement polarisées}

On rappelle au préalable la convention pour les variétés abéliennes CM que l'on utilise ici.
\begin{defin}\label{varabcm} Soit $A$ une variété abélienne de dimension $g$. On dit que $A$ est CM s'il existe une $\Q$-algèbre commutative semi-simple $F$ de dimension $2g$ et une injection
$$F\rightarrow \Q\otimes \operatorname{End} A.$$
On dit dans ce cas que $A$ a multiplication complexe par $F$. 

Dans le cas où $A_{\overline{K}}$ est CM on dit que $A$ est potentiellement CM.
\end{defin}

On commence par la proposition d'existence des variétés abéliennes CM qui nous intéressent dans la suite.
\begin{prop}\label{chap2existcm}
Soient $p$ un nombre premier impair, $\zeta_p$ une racine primitive $p$-ième de l'unité et $\ell>p$ un nombre premier. Il existe un corps de nombres $K$ et une variété abélienne $A$ sur $K$ de sorte que 
\begin{itemize}
\item[$(i)$] $\dim A= \frac{p-1}{2};$
\item[$(ii)$] $\End A\simeq \Z[\zeta_p];$
\item[$(iii)$] $A$ admet une polarisation principale.
\end{itemize}
De plus on peut choisir $K$ ramifié seulement au-dessus de $p$ et des places de caractéristiques résiduelles supérieures à $\ell$ avec $A$ de bonne réduction sur $K$. 
\end{prop}
\begin{proof}
On commence par trouver une variété abélienne $A$ sur $\C$ qui vérifie $(i),$ $(ii)$ et $(iii)$. 

 On note $\alpha=\frac{\zeta_p-\zeta_p^{-1}}{p}\in \Q(\zeta_p)$. Pour tout plongement $\tau\colon \Q(\zeta_p)\rightarrow \C$ l'image $\tau(\alpha)$ de $\alpha$ est un imaginaire pur que l'on note ${{\mathrm{i}}}\beta_{\tau}$. Alors 
$$\Phi=\{\tau \colon \Q(\zeta_p)\rightarrow \C \mid \beta_{\tau} >0\}$$
 définit un type CM de $\Q(\zeta_p)$ qui est primitif. Il suffit de voir pour cela que 
$$G=\{\sigma\in \mathrm{Gal}(\Q(\zeta_p) /\Q) \mid \Phi \sigma= \Phi\}=\{\mathrm{id}\}$$
par la proposition 26 de \cite{shimura}.
Si l'on fait correspondre à chaque élément de ce groupe de Galois son image de $\zeta_p$, $\Phi$ correspond aux racines primitives $p$-èmes de l'unité sur le demi-cercle supérieur et la composition par un élément $\sigma$ à une rotation. Une telle rotation non triviale ne laisse pas stable le demi-cercle supérieur ce qui donne bien $G=\{\mathrm{id}\}$. 

Le type CM primitif $\Phi$ définit une structure complexe sur $\Q(\zeta_p)\otimes \R$. On va maintenant montrer que le tore complexe $A= \Q(\zeta_p)\otimes \R/ \Z[\zeta_p]$ a les propriétés voulues. La dimension de $A$ est bien $\frac{p-1}{2}$ ce qui donne $(i)$. Pour $(x,y)\in (\Q(\zeta_p)\otimes \R)^2$ on pose
$$H(x,y)=2 \sum\limits_{\tau\in  \Phi} \beta_{\tau} \tau(x) \overline{ \tau(y)}.$$
On vérifie suivant \cite{mum} p. 212 que $H$ est une forme de Riemann avec 
$$\Im H ( \Z[\zeta_p], \Z[\zeta_p])=\Z.$$
 Plus précisément on calcule la matrice de $\Im H$ dans la base $(1,\zeta_p,\dots, \zeta_p^{p-2})$. On a 
$$H(\zeta_p^m, \zeta_p^n)= \Tr_{\Q(\zeta_p)/\Q} (\alpha \zeta_p^{m-n})$$
ce qui vaut $0$ si $|m-n|\neq 1$, $1$ si $m=n+1$ et $-1$ sinon. Cela donne pour matrice de $\Im H$
$$\begin{pmatrix}
0 & 1 & & 0\\
-1 & 0 & \ddots & \\
 & \ddots & \ddots    & 1 \\
0& & -1 & 0
 \end{pmatrix} $$
 qui a donc déterminant $1$. On obtient que le tore complexe $A$ est une variété abélienne et $H$ définit une polarisation principale. Pour $(ii)$ on a l'inclusion $\Z[\zeta_p] \hookrightarrow \End A$ et l'égalité vient du fait que $A$ est simple car $\Phi$ est primitif.

 Le corps réflexe de $(\Q(\zeta_p), \Phi)$ est encore $\Q(\zeta_p)$ car $\Q(\zeta_p)$ est une extension abélienne et $\Phi$ est un type primitif. Par le théorème p. 112 de \cite{shimura} le corps des modules de $A$ muni de ses endomorphismes et de sa polarisation est une extension de $\Q(\zeta_p)$ incluse dans son corps de classes de Hilbert. De la même façon le théorème 2 p. 118 de \cite{shimura} permet de calculer le corps des modules $K$ de cette structure à laquelle on a rajouté la $p$-torsion de $A$ comme corps de classes sur le corps réflexe et de borner son conducteur. On peut alors vérifier que $K$ n'est ramifié qu'en l'unique idéal au-dessus de $p$ de $\Z[\zeta_p]$. Finalement, le théorème 21.1 de \cite{shimura} assure que $K$ est un corps de définition de $A$.
 
 La variété abélienne $A$ a sa $p$-torsion définie sur $K$ donc par la proposition 4.7 de l'exposé IX de \cite{sga} elle a réduction semi-stable en toutes les places non archimédiennes de $K$ qui ne sont pas au-dessus de $p$. Par ailleurs comme $A$ est CM, par le théorème 6. a) de \cite{serretate} elle a bonne réduction potentielle donc bonne réduction en ces places. Quitte à faire une extension finie $L/K$ ramifiée seulement en les places au-dessus de $p$ et des places de caractéristiques résiduelles supérieures à $\ell$ on a $A_L$ a bonne réduction sur $L$ (voir paragraphe \ref{chap2rappel}).
\end{proof}

\begin{rem}
En particulier la variété abélienne $A$ du lemme précédent est CM et a multiplication complexe par $\Q(\zeta_p)$. 
\end{rem}

Soit $(A,\lambda)$ une variété abélienne principalement polarisée sur un corps $K$. Soit $B$ une variété abélienne sur $K$ telle qu'il existe une extension finie galoisienne $L$ et un isomorphisme $\varphi\colon B_L\rightarrow A_L$. Alors, comme la construction de la variété duale est fonctorielle, on a un isomorphisme $\varphi^{\vee}\colon A_L^{\vee}\rightarrow B_L^{\vee}$.  Il suit que $B_L$ a une polarisation principale $\varphi^*\lambda$, i.e. un isomorphisme $B_L\simeq B_L^{\vee}$ donné par $(\varphi^{\vee}) \circ \lambda \circ \varphi.$ On dispose par ailleurs de l'involution de Rosati correspondant à $\lambda$ sur $\End A_L \otimes \Q$ définie par ${}^{\dag} \colon f\mapsto \lambda^{-1} \circ f^{\vee} \circ \lambda$.  On donne maintenant un critère pour que $\varphi^*\lambda$ provienne d'une polarisation principale de $B$. On note $c$ le cocycle qui représente la classe de $B$ dans $H^1(\mathrm{Gal}(L/K), \operatorname{Aut} A_L)$, c'est-à-dire
$$\begin{array}{cccc}
c\colon & \mathrm{Gal}(L/K) & \longrightarrow & \operatorname{Aut} A_L \\
& \sigma & \longmapsto &  \varphi\circ \sigma(\varphi)^{-1}.
\end{array}$$

\begin{lem}\label{chap2lempolarprincip}
La polarisation principale $\varphi^* \lambda$ provient de $B$ si et seulement pour tout $\sigma\in \mathrm{Gal}(L/K)$ l'égalité
$$c(\sigma)^{\dag} c(\sigma)=\mathrm{id}$$
est vérifiée. 
\end{lem}
\begin{proof}
On remarque d'abord que $\varphi^* \lambda$ provient de $B$ si et seulement si elle est fixée par l'action de Galois, c'est-à-dire si et seulement si
$$\sigma(\varphi^* \lambda)= \varphi^* \lambda$$
pour tout $\sigma \in \mathrm{Gal}(L/K)$. 

Or pour $\sigma\in \mathrm{Gal}(L/K),$ on a
\begin{align*}
\sigma(\varphi^* \lambda)&=\sigma( \varphi^{\vee} \circ \lambda \circ \varphi) \\
&= \sigma(\varphi^{\vee}) \circ \sigma(\lambda) \circ \sigma(\varphi).\\
\end{align*}
L'égalité $\sigma(\varphi^* \lambda)= \varphi^* \lambda$ est donc équivalente à
$$\sigma(\varphi^{\vee}) \circ \sigma(\lambda) \circ \sigma(\varphi)=\varphi^{\vee}\circ \lambda \circ \varphi$$
soit à
$$\lambda = c(\sigma)^{\vee} \circ \lambda \circ c(\sigma)$$
et finalement à
$$c(\sigma)^{\dag} c(\sigma)=\mathrm{id}.$$
\end{proof}

On termine cette partie par un lemme sur les variétés abéliennes isotypiques qui sera utile en partie 5.  $A_{\overline{K}}.$
\begin{lem} \label{chap2corpsCM}Soient $A$ une variété abélienne sur un corps de nombres $K$ telle que $A_{\overline{K}}$ est isotypique et CM . Soit $K_A$ le corps de définition des endomorphismes de $A_{\overline{K}}$. Alors la variété abélienne $A'=A_{K_A}$ est isotypique. Il existe une variété abélienne $B$ de dimension $d$ sur $K_A$ telle que $A'$ est isogène à $B^h$ pour un entier $h\geq 0$. De plus on a 
$$\Q\otimes \operatorname{End} B\simeq Z$$
où $Z$ est un corps CM de dimension $2d$ sur $\Q$ avec $2dh=2g$.
\end{lem}
\begin{proof}
Comme $A_{\overline{K}}$ est isotypique l'algèbre $\operatorname{End} A_{\overline{K}} \otimes \Q=\operatorname{End} A'\otimes \Q$ est simple ce qui montre que $A'$ est isotypique. Il existe donc une variété abélienne simple $B$ sur $K_A$ telle que $A$ est isogène à $B^h$ pour un certain $h\geq 0$. La proposition 1.3.2.1 et le théorème 1.3.4 de \cite{conrad} donnent le résultat.
\end{proof}
\section{Le problème de Grunwald pour certains produits en couronne}

Soient $G$ un groupe fini, $k$ un corps de nombres et $S$ un ensemble fini de places de $k$. Le problème de Grunwald consiste à trouver une extension galoisienne $L/k$ de groupe $G$ et de comportement local prescrit aux places de $S$. Pour plus de détails à ce sujet, on renvoie le lecteur à la partie 2 de \cite{checco}.

 Notre but est de montrer que certains problèmes de Grunwald pour les groupes $\Z/p^m\Z \wr \mathfrak{S}_n$, où $p,m$ et $n$ varient, sont résolubles quitte à grossir le corps de base $k$ dans un premier temps. Dans un second temps, on s'intéressera au cas, plus difficile, du groupe $(\Z/4\Z\wr\mathfrak{S}_n) \rtimes\Z/2\Z$.
 
 On commence pour cela par rappeler la notion de produit en couronne de groupes.

\begin{defin} Soient $G$ un groupe fini et $n$ un entier naturel. On note $G\wr \mathfrak{S}_n$ le produit semi-direct de $G^n$ et $\mathfrak{S}_n$ où le groupe symétrique $\mathfrak{S}_n$ agit par permutations sur $G^n$.
\end{defin}

Pour un groupe fini $G,$ on a par définition
$$\operatorname{Card} G\wr \mathfrak{S}_n= n! (\operatorname{Card} G)^n.$$

\begin{lem} \label{cardcour} Soient $p$ un nombre premier impair et $n$ un entier naturel. On a 
$$v_p (\Card \Z/p\Z \wr \mathfrak{S}_n)= n+\sum\limits_{k=1}^{\infty} \bigg \lfloor \frac{n}{p^k}\bigg \rfloor$$
et
$$v_2(\Card \Z/4\Z \wr \mathfrak{S}_n)=2n+\sum\limits_{k=1}^{\infty} \bigg \lfloor \frac{n}{2^k}\bigg \rfloor.$$
\end{lem} 
\begin{proof}
Cela découle directement de la définition et de la formule de Legendre
$$v_p (n!)=\sum\limits_{k=1}^{\infty} \bigg \lfloor \frac{n}{p^k}\bigg \rfloor$$
pour tout nombre premier $p$.
\end{proof}

\begin{rem}
On remarque que si, pour un entier naturel non nul $g$ fixé et un nombre premier impair $p$, on choisit $n=\lfloor \frac{2g}{p-1} \rfloor,$ alors 
$$v_p (\Card \Z/p\Z \wr \mathfrak{S}_n)=r(2g,p).$$
Pour le nombre premier $2,$ on a seulement
$$v_2(\Card \Z/4\Z \wr \mathfrak{S}_g)= r(2g,2)-g$$
 en prenant $n=g$.
 
Les groupes $\Z/p\Z$ et $\Z/4\Z$ correspondent aux groupes des racines de l'unité dans $\Z[\zeta_p]$ et $\Z[\mathrm{i}]$ respectivement. Si l'on disposait d'une courbe elliptique avec le groupe des quaternions comme groupe d'automorphismes alors le calcul serait ici
$$v_2(\Card Q_8 \wr \mathfrak{S}_g)= r(2g,2).$$

\end{rem}

\subsection{Les groupes $\Z/p^m\Z\wr \mathfrak{S}_n$}

Dans cette sous-partie, on traite le cas des groupes $\Z/p^m\Z \wr \mathfrak{S}_n$, où $p,n$ et $m$ varient. La résolution du problème de Grunwald pour ces groupes découlera des travaux de Saltman dans \cite{salt} après quelques propositions de préparation. On construit dans un premier temps, pour un corps de nombres $k$, une extension $K$ de $k$ sur laquelle on montrera que le problème est résoluble.

On commence par deux résultats\footnote{Le deuxième résultat et sa preuve nous ont été communiqués par Akio Tamagawa, ce qui permet, avec la maîtrise  de la ramification locale, une présentation plus élégante de nos résultats.}
 standards.

\begin{prop} \label{locglob} Soient $K$ un corps de nombres et $v_1,\dots, v_n$ des places ultramétriques de $K$. Soient $K_i'/K_{v_i}$ des extensions finies ayant même degré $d$. Alors il existe une extension $K'/K$ avec une unique place $w_i|v_i$ pour chaque $i$ telle que $K'_{w_i}\simeq K_i'$ et $[K':K]= d$. 
\end{prop}

Ce résultat est obtenu de façon identique à d'autres résultats classiques, voir par exemple le chapitre 6 de \cite{rib} et particulièrement le théorème 4. La démonstration se base sur la continuité des racines sur les corps locaux et le théorème de l'élément primitif qui fournit un polynôme irréductible à approcher pour chaque extension locale.

\begin{lem}\label{tam}
Soient $K$ un corps de nombres et $v$ une place finie de $K$ au-dessus de $p$. Le pro-$p$ quotient maximal $H$ du groupe de Galois absolu $G$ de $K_v^{\mathrm{nr}}$ est un pro-$p$ groupe libre de rang infini dénombrable.
\end{lem}
\begin{proof}
Par les théorèmes 9.3 et 9.7 de \cite{koch} le groupe $H$ est pro-$p$ libre. Par ailleurs, la suite spectrale d'Hochschild-Serre donne un isomorphisme
$$H^1(H,\F_p)\simeq H^1(G,\F_p)$$
avec $H^1(G,\F_p)=\varinjlim H^1(G_n,\F_p)$ où $G_n$ est le groupe de Galois absolu de l'unique extension non ramifiée de degré $n$ de $K_v$. Par les théorèmes 3.9.1 et 7.5.11 de \cite{neuk} on a de plus $\dim H^1(G_n,\F_p)\geq n$ ce qui donne que la dimension de $H^1(H,\F_p)$ est infinie dénombrable. Le théorème 3.9.1 de \cite{neuk} permet alors de conclure que $H$ est de rang infini dénombrable.
\end{proof}

\begin{lem}\label{extlocale} Soit $G$ un $p$-groupe fini. Alors pour tout corps de nombres $k$ et toute place $v$ de corps résiduel de caractéristique $p$ de $k,$ il existe une extension finie non ramifiée $K^{(v)}$ de $k_v$ telle que $G$ est le groupe de Galois d'une extension totalement ramifiée $L^{(v)}/K^{(v)}$. 
\end{lem}
\begin{proof}
Par le lemme \ref{tam}, on dispose d'une extension galoisienne de $k_v^{\mathrm{nr}}$ de groupe de Galois $H$ un pro-$p$ groupe libre de rang infini dénombrable. Soit $(s_n)_{n\in \N}$ une base de ce groupe. On définit un morphisme surjectif $H\rightarrow G$ en choisissant des images presque toutes égales à $1_G$ des $s_n$. Ceci détermine une extension galoisienne de $k_v^{\mathrm{nr}}$ de groupe $G$. Cette extension descend en une extension galoisienne $L^{(v)}/K^{(v)}$ de même groupe pour un choix convenable de $K^{(v)}\subset k_v^{\mathrm{nr}}$ extension finie de $k_v$. 
\end{proof}

Le résultat de préparation avant de résoudre notre problème de Grunwald peut alors s'énoncer comme suit.

\begin{lem}\label{extglobloc} Soient $G$ un groupe fini, $k$ un corps de nombres et $S$ un ensemble fini de places ultramétriques de $k$. Il existe une extension finie $K/k$ telle que pour toute place $v$ de $S$
\begin{itemize}
\item[(1)] il existe une unique place $w$ de $K$ au-dessus de $v$;
\item[(2)] l'extension $K_w/k_v$ est non ramifiée;
\item[(3)] il existe une extension galoisienne $L^{(v)}/K_w$ totalement ramifiée;
\item[(4)] le groupe $\mathrm{Gal}(L^{(v)}/K_w)$ est isomorphe à un $p$-Sylow de $G$ où $p$ est la caractéristique résiduelle de $v$.
\end{itemize}
\end{lem}
\begin{proof}
Par le lemme \ref{extlocale} on obtient des extensions finies non ramifiées $K^{(v)}/k_{v}$ pour les places $v$ de $S$ et des extensions $L^{(v)}/K^{(v)}$ totalement ramifiées galoisiennes de groupes un $p$-Sylow de $G$.  Quitte à faire des extensions non ramifiées des $K^{(v)}$, et donc linéairement disjointes des $L^{(v)}$, on peut supposer que les degrés $[K^{(v)}:k_{v}]$ sont égaux. La proposition \ref{locglob} appliquée à cette situation donne l'existence de $K$.
\end{proof}

\begin{prop} \label{grun} Soient $k$ un corps de nombres, $p$ un nombre premier et $m,n$ des entiers naturels non nuls. Si $p=2$ on suppose que $m=1$ ou $m= 2$. On note $G$ le groupe $\Z/p^m\Z\wr \mathfrak{S}_n$. Soit $S$ un ensemble fini de places ultramétriques de $k$ contenant une place $v$ au-dessus de $p$. Alors il existe une extension finie $K/k$ et une extension $L/K$ galoisienne de groupe $G$ telles que
\begin{itemize}
\item[(1)] au-dessus de chaque place de $S,$ il existe une unique place de $K$ et elle est non ramifiée;

\item[(2)] le groupe d'inertie de $\mathrm{Gal}(L/K)$ en l'unique place $w$ de $K$ au-dessus de $v$ est un $p$-Sylow de $G$.
\end{itemize}
\end{prop}
\begin{proof}

On peut supposer sans perte de généralité que $S$ contient des places au-dessus de tous les diviseurs premiers de $\Card G$. Le lemme \ref{extglobloc} appliqué avec $k,G$ et $S$ fournit une extension $K/k$ et des extensions locales. Les résultats de Saltman dans \cite{salt} permettent de réaliser ces extensions comme complétés d'une extension $L/K$. En effet d'après les théorèmes 2.1 et 5.1 de \cite{salt}, les groupes $\Z /p^m\Z$ et $\mathfrak{S}_n$ admettent des extensions galoisiennes génériques sur $K$ (par choix de $m$ pour $p=2$) et donc le produit en couronne $G=\Z/p^m\Z\wr \mathfrak{S}_n$ aussi par le théorème 3.3 de \cite{salt}. Le théorème 5.9 de \cite{salt} s'applique donc avec des sous-groupes $H_i$ qui sont des sous-groupes de Sylow de $G$ et pour tout nombre premier $\ell$ au moins un $\ell$-Sylow fait partie des $H_i$ par choix de $S$.

\end{proof}

\subsection{Le groupe $(\Z/4\Z\wr\mathfrak{S}_n) \rtimes\Z/2\Z$} \label{chap2pargrun2}

On fixe pour cette partie un entier naturel non nul $n$. On voit le produit en couronne $G=\Z/4\Z\wr \mathfrak{S}_n$ comme un groupe de matrices (voir \cite{sheptod} et le lemme 2.3 de \cite{r}) plongé dans $\mathrm{GL}_n(\Q(\mathrm{i}))$ et $H=G\rtimes \Z/2\Z$ est obtenu par l'action de la conjugaison complexe sur $G$. On a le diagramme d'inclusions
$$\xymatrix{
H \ar@{}[r]|-*[@]{\subset} & \operatorname{Aut}_{\Q} \Q(\mathrm{i})^n \\
G \ar@{}[r]|-*[@]{\subset} \ar@{}[u]|-*[@]{\subset} & \mathrm{GL}_n(\Q(\mathrm{i})) \ar@{}[u]|-*[@]{\subset} . 
}$$
Cela permet de voir $H$ comme le produit direct d'ensembles $G\times \{\mathrm{id}, \gamma\}$ où $\gamma\in \operatorname{Aut}_{\Q} \Q(\mathrm{i})$ est la conjugaison complexe et on note un élément de $H$ par un couple $(g,\gamma)$ ou $(g, \mathrm{id})$ avec $g\in G$. Pour $g\in G$ on utilise la notation 
$$\overline{g}=\gamma\cdot g=(1,\gamma)(g,\mathrm{id})(1,\gamma).$$

 Notre but est ici de résoudre un problème de Grunwald pour $H$ sur un corps de nombres $K$ de manière à ce que $K(\mathrm{i})=L^G$ où $L/K$ est galoisienne de groupe $H$ et un $2$-Sylow de $H$ est un groupe d'inertie de $L$. L'action de la conjugaison complexe sur $G$ est d'ordre $2$ et donc il existe un $2$-Sylow stabilisé par celle-ci. On peut donc choisir $H_2$, un $2$-Sylow de $H$, de la forme $G_2\rtimes \Z/2\Z$ où $G_2$ est un $2$-Sylow de $G$. 

On commence par un lemme technique qui correspond au lemme \ref{extlocale} dans le cas particulier que l'on traite ici.

\begin{lem} \label{extloc2} Soit $M/\Q_2^{\mathrm{nr}}(\mathrm{i})$ une extension galoisienne finie totalement sauvagement ramifiée.  Il existe une extension finie $K_2/\Q_2$ non ramifiée et une extension galoisienne totalement ramifiée $L_2/K_2$ de groupe $H_2$ vérifiant $L_2^{G_2}=K_2(\mathrm{i})$ et $L_2 \Q_2^{\mathrm{nr}} \cap M =\Q_2^{\mathrm{nr}}(\mathrm{i}).$
\end{lem}
\begin{proof}
Par le lemme \ref{tam} on dispose d'une extension galoisienne $L/\Q_2^{\mathrm{nr}}$ de groupe de Galois $V$ un pro-$2$ groupe libre de rang infini dénombrable. L'extension $\Q_2^{\mathrm{nr}}(\mathrm{i})$ définit un sous-groupe $\widetilde{V}$ d'indice $2$ de $V$. Soit $(s_k)_{n\in \N}$ une base de $V$. Le groupe $V_{\mathrm{lib}}=<s_k \mid k\in \N>$ est un groupe libre inclus dans $V$ dont l'adhérence est $V$ et $\widetilde{V}\cap V_{\mathrm{lib}}=\widetilde{V}_{\mathrm{lib}}$ est un sous-groupe d'indice $2$ de $V_{\mathrm{lib}}$ dont l'adhérence est $\widetilde{V}$. Il existe donc $i\in \N$ tel que $s_i\notin \widetilde{V}_{\mathrm{lib}}.$ Quitte à renuméroter on peut supposer $i=0$. Pour $k\geq 1$, si $s_k\notin \widetilde{V}_{\mathrm{lib}}$ alors $s_k s_0 \in  \widetilde{V}_{\mathrm{lib}}$ et donc quitte à remplacer $s_k$ par $s_k s_0$ on peut supposer que $s_k \in \widetilde{V}_{\mathrm{lib}}$ pour $k\geq 1$. On vérifie alors que la famille $\{s_0^2,s_k, s_0s_ks_0^{-1}\ \mid k\geq 1\}$ est une base de $\widetilde{V}_{\mathrm{lib}}$, ce qui peut aussi se déduire du théorème de Nielsen-Schreier.  En raisonnant comme dans la preuve du théorème 5.4.4 de \cite{Wils} on obtient de plus que $\widetilde{V}$ est un pro-$2$ groupe libre sur la même base. 

On est maintenant en mesure de construire l'extension $L_2$. Comme $M$ est galoisienne elle est donnée par un morphisme surjectif $f\colon\widetilde{V}\rightarrow F$ dont les images de la base $\{s_0^2,s_k, s_0s_ks_0^{-1}\ \mid k\geq 1\}$ sont presque toutes égales à $1_F$. Soit $N$ un entier tel que pour tout $k\geq N$, $f(s_k)=f(s_0s_ks_0^{-1})=1_F$. On définit alors un morphisme surjectif $\psi \colon \widetilde{V}\rightarrow G_2$ en choisissant des images de la base $\{s_0^2,s_k, s_0s_ks_0^{-1}\ \mid k\geq 1\}$ de la façon suivante.  On choisit pour chaque élément de $g\in G_2\setminus \{1_{G_2}\}$ un antécédent $s_k$ avec $k>N$ et on impose $\psi(s_0^2)=\psi(s_m)=1_{G_2}$ pour les autres. On impose de plus $\psi(s_0s_ks_0^{-1})=\overline{\psi(s_k)}$ pour tout $k$.  On étend ce morphisme en un morphisme surjectif $\varphi\colon V\rightarrow H_2$ en choisissant des images de la base $\{s_k \mid k\in \N\}$ de façon compatible. Pour cela on pose $\varphi(s_0)=(1_G,\gamma)$ et $\varphi(s_k)=(\psi(s_k),\mathrm{id})$ si $k\geq 1$. Le morphisme $\varphi$ définit une extension $L_2^{\mathrm{nr}}/\Q_2^{\mathrm{nr}}$ galoisienne de groupe $H_2$ dont le corps fixé par $G_2$ est $\Q_2^{\mathrm{nr}}(\mathrm{i})$ par construction et qui vérifie $L_2^{\mathrm{nr}}\cap M=\Q_2^{\mathrm{nr}}(\mathrm{i})$ car $(f,\psi)\colon \widetilde{V}\rightarrow F\times G_2$ est surjectif.  

\`A nouveau on peut descendre cette extension en une extension galoisienne totalement ramifiée $L_2$ de $K_2$, une extension finie non ramifiée de $\Q_2,$ telle que $L_2^{G_2}=K_2(\mathrm{i})$ et $L_2\Q_2^{\mathrm{nr}}=L_2^{\mathrm{nr}}$.  
\end{proof}

Par la proposition \ref{locglob} on dispose alors pour tout choix de $M$ comme dans l'énoncé d'une extension $K/\Q$ non ramifiée en $2$ et avec une unique place $v$ au-dessus de $2$ telle qu'on ait une extension galoisienne $L_2/K_{v}$ totalement ramifiée de groupe $H_2$ et dont le corps fixé par $G_2$ est $K_{v}(\mathrm{i})$. Ceci sera utilisé pour le résultat principal de cette partie avec le lemme de géométrie algébrique suivant qui permet de remplacer les résultats de Saltman.
\medskip

\begin{lem}\label{chap2ratio} Soit $K$ un corps de nombres ne contenant pas $\mathrm{i}$. On fait agir le groupe $H$ par $K$-automorphismes sur $R=K(\mathrm{i})[(X_g)_{g\in G}]$ de manière à avoir $(h,\gamma)\cdot aX_g= \overline{a} X_{h\overline{g}}$ pour $h,g \in G$ et $a\in K(\mathrm{i})$.  
Alors il existe des polynômes algébriquement indépendants $(P_a)_{a\in \mA}$ de $R$ et un polynôme $P\in K[(P_a)_{a\in \mA}]\setminus\{0\}$ tels que
$$R[\frac{1}{P}]^H=K[(P_a)_{a\in \mA}, \frac{1}{P}].$$

De plus, on peut choisir $P$ de manière à ce que l'extension $R[\frac{1}{P}]/R[\frac{1}{P}]^H$ soit étale. 
\end{lem}

\begin{proof}
On considère la représentation régulière de $G$ sur $K(\mathrm{i})$ donnée par
$$\operatorname{Reg} = \bigoplus\limits_{g\in G} K(\mathrm{i}) X_g.$$ 
Le sous-espace $V$ engendré par la famille de vecteurs $Y_i=\sum\limits_{g\in G} \overline{g_{i,1}} X_g$, $1\leq i\leq n$ est une sous-représentation de $\operatorname{Reg}$ isomorphe à la représentation standard (donnée par $G\subset \mathrm{GL}_n(K(\mathrm{i}))$). En effet, pour $h\in G$ on a
\begin{align*}
h\cdot Y_i&= \sum\limits_{g\in G} \overline{g_{i,1}} X_{hg}\\
 &= \sum\limits_{g\in G} \overline{(h^{-1}g)_{i,1}} X_g.\\
 \end{align*}
 Or la matrice de $h^{-1}$ est la conjuguée de la transposée de $h$ d'où
\begin{align*}
h\cdot Y_i&= \sum\limits_{g\in G} \sum\limits_{j=1}^n h_{j,1}\overline{g_{j,1}} X_g\\
&= \sum\limits_{j=1}^n h_{j,1} Y_j.
 \end{align*}
 Soient $W$ un supplémentaire de $V$ dans $\operatorname{Reg}$ et $L=K(\mathrm{i})(Y_1,\dots , Y_n)$. Par le lemme 3.5 de \cite{cosan} on a
 $$W\otimes L= (W\otimes L)^G\otimes_{L^G} L.$$
 Soit $(\widetilde{F}_j)_{j\in J}$ une base de $(W\otimes L)^G$ sur $L^G$. Comme on a une injection
 $$(W\otimes L)^G \longrightarrow (\operatorname{Reg} \otimes L)^G$$
 on peut écrire $\widetilde{F}_j=\sum\limits_{g\in G} g(P_j)\otimes X_g$ avec $P_j\in L$. 
 
 Maintenant on a $K(\mathrm{i})((X_g)_{g\in G})=L(W\otimes L)=L((F_j)_{j\in J})$ avec $F_j$ l'image de $\widetilde{F}_j$ par $P\otimes X_g\mapsto PX_g$. Il suit l'égalité $K(\mathrm{i})((X_g)_{g\in G})^G =L^G(F_j)$. On peut par ailleurs vérifier que $L^G$ est engendré par les fonctions symétriques en les $Y_1^4,\dots, Y_n^4$.
  D'après le théorème 1(c) de \cite{bourbaki}  IV.6.1 la famille $\{Y_1^{4b_1}\cdots Y_n^{4b_n} \mid 0\leq b_i \leq n-i\}$ est une base de $K(\mathrm{i})(Y_i^4)$ sur $L^G$ donc la famille $\{Y_1^{a_1}\cdots Y_n^{a_n} \mid a\in \mA\}$ avec $\mA=\{ a\in \N^n \mid a_i\leq 4(n-i)+3\}$ est une base de $L$ sur $L^G$. Il suit que chaque $F_j$ est une combinaison linéaire sur $L^G$ des éléments de la famille
 $$\{P_a=\sum\limits_{g\in G} g(Y_1^{a_1}\cdots Y_n^{a_n}) X_g\}_{a\in \mA}.$$
 
 Ces éléments étant stables par $G$ on a des inclusions
 $$K(\mathrm{i})((X_g)_{g\in G})^G=L^G(F_j)\subset L^G(P_a)\subset K(\mathrm{i})((X_g)_{g\in G})^G$$
 d'où l'égalité $K(\mathrm{i})((X_g)_{g\in G})^G=L^G((P_a)_{a\in \mA})$.
 
 De plus, le calcul suivant montre que les $P_{(4k-1,0,\dots, 0)}$ pour $1\leq k\leq n$ engendrent $K(\mathrm{i})[Y_i]^G$ par les formules de Newton :
\begin{align*}
P_{(4k-1,0,\dots, 0)}&= \sum\limits_{g\in G} g\cdot(Y_1^{4k-1}) X_g \\
&= \sum\limits_{g\in G} (\sum\limits_{i=1}^n g_{i,1} Y_i)^{4k-1} X_g \\
&= \sum\limits_{g\in G} \sum\limits_{i=1}^n \overline{g_{i,1}} Y_i^{4k-1}X_g \\
&= \sum\limits_{i=1}^n Y_i^{4k-1} \sum\limits_{g\in G} \overline{g_{i,1}} X_g \\
&= \sum\limits_{i=1}^n Y_i^{4k}.
\end{align*}
Il suit $K(\mathrm{i})((X_g)_{g\in G})^G =K(\mathrm{i})((P_a)_{a\in \mA})$ et on en déduit l'indépendance algébrique de la famille $(P_a)_{a\in \mA}$ car $\operatorname{degtr} K(\mathrm{i})((X_g)_{g\in G})= \Card G= \Card \mA$. Finalement, il existe $P\in K(\mathrm{i})[(P_a)_{a\in \mA}]\setminus\{0\},$ que l'on peut choisir stable par conjugaison, tel que
$$R[\frac{1}{P}]^G=K(\mathrm{i})[(P_a)_{a\in \mA}, \frac{1}{P}]$$
et comme les $P_a$ sont stables par conjugaison on a même
$$R[\frac{1}{P}]^H=K[(P_a)_{a\in \mA}, \frac{1}{P}].$$

Quitte à rajouter des facteurs à $P$ de manière à ce que l'extension soit étale on a le résultat annoncé. 
\end{proof}

On peut finalement déduire des propositions précédentes le résultat de type Grunwald que l'on veut.

\begin{theo}\label{extgrun2} Soit $M/\Q_2^{\mathrm{nr}}(\mathrm{i})$ une extension galoisienne finie totalement sauvagement ramifiée. Il existe un corps de nombres $K$ non ramifié en $2$ et une extension galoisienne $L/K$ de groupe $H$ telle que
\begin{itemize}
\item[(1)] le groupe d'inertie en une place $w$ de $L$ au-dessus de $2$ est $H_2$;
\item[(2)] $L_{w}^{G_2}=K_v(\mathrm{i})$ où $v=w_{|K};$
\item[(3)] $L_{w} \Q_2^{\mathrm{nr}}\cap M=\Q_2^{\mathrm{nr}}(\mathrm{i}).$  
\end{itemize}
\end{theo}
\begin{proof}

Par le lemme \ref{extloc2} appliqué avec $M$ on dispose d'une extension locale $L_2/K_2$ telle que $L_2^{G_2}=K_2(\mathrm{i})$ et par la proposition \ref{locglob} d'un corps de nombres $K$ et d'une place $v$ de $K$ au-dessus de $2$ telle que $K_v=K_2$. 

Maintenant, le lemme \ref{chap2ratio} appliqué avec $K$ fournit un morphisme fini étale galoisien de groupe $H$
$$\pi\colon W\longrightarrow U$$
où $W=\Spec K(\mathrm{i})[(X_g)_{g\in G}, \frac{1}{P}]$ et $U=\Spec K[(P_a)_{a\in \mA}, \frac{1}{P}]$ avec les notations du lemme. 

On pose $\mathbf{L}_2=\operatorname{Ind}^H_{H_2}L_2$ la $K_2$-algèbre galoisienne de groupe $H$ produit de $[H:H_2]$ copies de $L_2$ et vérifiant
$$\mathbf{L}_2^G=L_2^{G_2}=K_2(\mathrm{i}).$$

Par le lemme 5.2 de \cite{salt} appliqué au polynôme $\widetilde{P}\in K[(X_h)_{h\in H}]$ déduit de $P$ par la substitution $X_g=X_{g,\mathrm{id}}+X_{g,\gamma}$ il existe une base normale $(\xi_h)_{h\in H}$ de $\mathbf{L}_2$ sur $K_2$ qui n'annule pas $\widetilde{P}$. La famille $(w_g)_{g\in G}$ définie par
$$w_g=\xi_{g,\mathrm{id}}+\xi_{g,\gamma}$$
est alors une base normale de $\mathbf{L}_2$ sur $K_2(\mathrm{i})$ qui vérifie de plus
$$\overline{w_g}=w_{\overline{g}}.$$
En effet on a, pour $g\in G$, $\overline{\xi_{g,\mathrm{id}}}=\xi_{\overline{g},\gamma}$, $\overline{\xi_{g,\gamma}}=\xi_{\overline{g},\mathrm{id}}$ par définition du produit semi-direct et du fait que les $(\xi_h)_{h\in H}$ forment une base normale de $\mathbf{L}_2$ sur $K_2$. Par construction on a de plus $P((w_g)_{g\in G})=\widetilde{P}((\xi_h)_{h\in H})\neq 0$. On définit donc une spécialisation qui respecte l'action de $H$ par $X_g\mapsto w_g$ de manière à ce que $P(w_g)\neq 0$. Cela induit un diagramme commutatif

$$\begin{tikzcd}
K(\mathrm{i})[(X_g)_{g\in G}, \frac{1}{P}] \ar[d] & K[(P_a)_{a\in \mA}, \frac{1}{P}] \ar[l] \ar[d, "\varphi_2"]\\
\mathbf{L}_2 &  K_2 \ar[l].
\end{tikzcd}$$ 
Le morphisme $\varphi_2$ détermine ainsi un $K_2$-point $z_2$ de $U$. La fibre de $z_2\in U(K_2)$ est $\Spec \mathbf{L}_2$ par construction.

 Il reste à voir qu'un point $K$-rationnel de $U$ assez proche de $z_2$ a une fibre connexe qui donne une extension de corps qui convient. Comme $U$ vérifie l'approximation faible on peut, par le théorème 1.3 de \cite{eke}, choisir un $K$-point $z\in U$ suffisamment proche de $z_2$ pour la topologie analytique sur $U(K_2)$ dont la fibre est connexe et assez proche de manière à avoir l'égalité $\pi^{-1}(\iota_2^* z)=\pi^{-1}(z_2)$ par le lemme 3.5.74 de \cite{poonen}, avec $\iota_2\colon \Spec K_2 \rightarrow \Spec K$. Soit $L$ le corps tel que $\pi^{-1}(z) =  \Spec L$. C'est une extension galoisienne de groupe $H$. Par ce qui précède on a un isomorphisme de $K_2$-algèbres galoisiennes
$$L\otimes K_2\simeq \mathbf{L}_2$$
ce qui assure le comportement local de $L$.  
\end{proof}

\section{Variétés abéliennes tordues}

On va maintenant utiliser les résultats des deux parties précédentes pour construire les variétés abéliennes de l'énoncé du théorème \ref{chap2main}. Notre méthode consiste à tordre une variété abélienne $A$, c'est-à-dire la remplacer par une variété abélienne $B$ sur $K$ telle que $B_{L}\simeq A_{L}$ pour une extension galoisienne finie $L/K$ (on dit que $B$ est une $L/K$-forme de $A$). On commence par étudier l'effet de ce procédé sur les groupes de monodromie finie. 

\subsection{Les résultats de torsion}

On généralise le théorème 4.3 de \cite{Sb2004} énoncé pour les corps locaux dans deux directions différentes. Tout d'abord on obtient une généralisation directe aux corps de nombres et dans un second temps on donne une version avec des hypothèses réduites.

 Dans la suite de cette partie, pour un corps de nombres $K$ et une place non archimédienne $v$ de $K,$ on fixe un choix de plongement $\overline{K}\hookrightarrow \overline{K_v}$ et donc une place $\overline{v}$ de $\overline{K}$ au-dessus de $v$. On note $I(\overline{v}/v)$ le groupe d'inertie de $\overline{K_v}/K_v$ vu comme sous-groupe du groupe de décomposition de $\mathrm{Gal}(\overline{K}/K)$ associé à $\overline{v}$.

\begin{theo} \label{twist} Soient $A$ une variété abélienne de dimension $g$ sur un corps de nombres $K$ ayant bonne réduction et $L/K$ une extension galoisienne finie. Pour toute place finie $v$ de $K,$ on note $I_{v}$ le sous-groupe d'inertie de la place $\overline{v}_{|L}$ de $L$ au-dessus de $v$. Soit
$$c\colon \mathrm{Gal}(L/K)\longrightarrow \mathrm{Aut} A_L$$
un morphisme injectif. Alors il existe une variété abélienne $B$ de dimension $g$ sur $K$ telle qu'on ait des isomorphismes $\Phi_{B,v}\simeq I_{v}$ pour toute place finie $v$ de $K$.
\end{theo}
\begin{proof}
On considère le morphisme
$$\tilde{c}\colon \mathrm{Gal}(\overline{K}/K) \rightarrow \mathrm{Aut}  A_{L}$$
obtenu en composant $c$ et la surjection $\mathrm{Gal}(\overline{K}/K)\rightarrow \mathrm{Gal}(L/K)$.  

Soient $B$ la variété abélienne obtenue en tordant $A$ par le cocycle $\tilde{c}$ et $\varphi$ l'isomorphisme $B_L \rightarrow A_L$ qui vérifie $\varphi \sigma(\varphi)^{-1} =\tilde{c}(\sigma)$ pour tout $\sigma\in \mathrm{Gal}(\overline{K}/K)$ (ce procédé est décrit page 131 de \cite{se2}). 

Soient $v$ une place finie de $K$ et $w= \overline{v}_{|L}$ la place de $L$ au-dessus de $v$ définie par $\overline{v}$.
On a $\Phi_{B,v}=I(\overline{v}/v)/I_{B,v}$ où $I_{B,v}=\{\sigma\in I(\overline{v}/v)~\mid~ \rho_{B,\ell}(\sigma)=1\}$ avec $\ell$ un nombre premier distinct de la caractéristique $p$ du corps résiduel de $v$. On obtient 
$$\rho_{B,\ell}(\sigma)= \sigma(\varphi^{-1}) \rho_{A,\ell} (\sigma) \varphi= \varphi^{-1} \tilde{c}(\sigma)\rho_{A,\ell}(\sigma) \varphi.$$
Il suit 
$$I_{B,v}= \{\sigma\in I(\overline{v}/v)~\mid~ \tilde{c}(\sigma) \rho_{A,\ell} (\sigma)=1\}.$$

Comme $A$ a bonne réduction on a $\rho_{A,\ell} (\sigma)=1$ ce qui donne $I_{B,v}=\Ker \tilde{c}_{|I(\overline{v}/v)}$. On montre maintenant l'égalité $\Ker \tilde{c}_{|I(\overline{v}/v)} = I(\overline{v}/w)$. Par définition de $\tilde{c}$ on a $\Ker \tilde{c}=\mathrm{Gal}(\overline{K}/L)$ car $\Ker c=\{1\}.$ Il suit $\Ker \tilde{c}_{|I(\overline{v}/v)}=\mathrm{Gal}(\overline{K}/L)\cap I(\overline{v}/v)=I(\overline{v}/w).$  On en déduit 
$$\Phi_{B,v}= I(\overline{v}/v)/I(\overline{v}/w)=I_{v}.$$ 
\end{proof}

L'énoncé précédent ne suffit pas pour construire des variétés abéliennes CM avec monodromie finie sauvage maximale en une place résiduelle de caractéristique $2$. Pour ce faire on s'affranchit de l'hypothèse de bonne réduction. Dans le cas d'une variété abélienne sur un corps local on utilise des notations analogues à celles introduites dans le paragraphe \ref{chap2rappel} sans la mention de la place en indice. 

\begin{theo} \label{twist2} Soient $A$ une variété abélienne sur un corps local $K$ à corps résiduel algébriquement clos et $L/K$ une extension galoisienne finie. On suppose que le corps $K_A$ de définition des endomorphismes de $A_{\overline{K}}$ est un sous-corps de $L$ et que $A$ a bonne réduction potentielle. On suppose de plus que $L\cap K_{A,s}=K_A$. Soit 
$$c\colon \mathrm{Gal}(L/K)\longrightarrow \mathrm{Aut} A_L$$
un cocycle tel que $c_{|\mathrm{Gal}(L/K_A)}$ est injectif. Alors la $L/K$-forme de $A$ associée au cocycle $c$ est une variété abélienne $B$ sur $K$ telle que
$$[L:K_A][K_{A,s}:K] \mid  [K_{B,s}:K]\Card \mC \mid [LK_{A,s}:K]\Card \mC$$
 où $\mC$ est le centralisateur de $\mathrm{Gal}(K_{A,s}/K_A)$ dans $\Phi_A=\mathrm{Gal}(K_{A,s}/K)$.
\end{theo}
\begin{proof}
On considère le morphisme
$$\tilde{c}\colon \mathrm{Gal}(\overline{K}/K) \rightarrow \mathrm{Aut}  A_{L}$$
obtenu en composant $c$ et la surjection $\mathrm{Gal}(\overline{K}/K)\rightarrow \mathrm{Gal}(L/K)$.  

Soient $B$ la variété abélienne obtenue en tordant $A$ par le cocycle $\tilde{c}$ et $\varphi$ l'isomorphisme $B_L \rightarrow A_L$ qui vérifie $\varphi \sigma(\varphi)^{-1} =\tilde{c}(\sigma)$ pour tout $\sigma\in \mathrm{Gal}(\overline{K}/K)$. 

On a $\Phi_B=\mathrm{Gal}(\overline{K}/K)/I_{B}$ où $I_{B}=\{\sigma\in \mathrm{Gal}(\overline{K}/K)\mid \rho_{B,\ell}(\sigma)=1\}$ avec $\ell$ un nombre premier distinct de la caractéristique $p$ du corps résiduel. On obtient par la formule de torsion, pour tout $\sigma\in \mathrm{Gal}(\overline{K}/K),$
$$\rho_{B,\ell}(\sigma)= \sigma(\varphi^{-1}) \rho_{A,\ell} (\sigma) \varphi= \varphi^{-1} \tilde{c}(\sigma)\rho_{A,\ell}(\sigma) \varphi.$$
Il suit 
$$I_{B}= \{\sigma\in \mathrm{Gal}(\overline{K}/K)\mid \tilde{c}(\sigma) \rho_{A,\ell} (\sigma)=1 \}.$$
Soit $\sigma\in I_{B,v}$. L'action de $\tilde{c}(\sigma)$ sur le module de Tate $\operatorname{T}_{\ell}A$ se fait par $K_A$-automorphismes. Il suit que $\tilde{c}(\sigma)$ vu comme automorphisme de $\operatorname{T}_{\ell}A$ commute aux éléments de $\mathrm{Gal}(K_{A,s}/K_A)\subset\operatorname{Aut} \operatorname{T}_{\ell}A $. La dernière inclusion vient de ce que $\rho_{A,\ell}$ induit une injection $\mathrm{Gal}(K_{A,s}/K)\subset  \operatorname{Aut} \operatorname{T}_{\ell}A $ par définition de $K_{A,s}$. Il suit que $\rho_{A,\ell}(\sigma)$ commute aux éléments de $\mathrm{Gal}(K_{A,s}/K_A)$ car $\rho_{A,\ell}(\sigma)=\tilde{c}(\sigma)^{-1}$ et est donc dans $\mC$.

Les hypothèses donnent le diagramme d'extensions suivant
$$
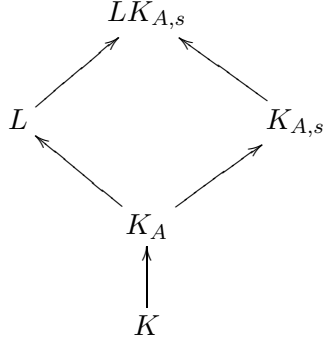
\xymatrix{
 & LK_{A,s} & \\
 L \ar[ur] & & K_{A,s} \ar[ul] \\
 & K_A \ar[ur] \ar[ul] & \\
 & K \ar[u] & 
}$$

On considère maintenant la restriction $(\rho_{A,\ell})_{|I_{B}} \colon I_{B}\rightarrow \mathrm{Gal}(K_{A,s}/K)$. Par ce qui précède son image est dans $\mC$ et son noyau est par construction $I_{v} \cap \Ker \rho_{A,\ell}$. On montre que cette intersection est égale à $\mathrm{Gal}(\overline{K}/LK_{A,s})$. Soit donc $\sigma\in I_{B} \cap \Ker \rho_{A,\ell}$. Comme $\Ker \rho_{A,\ell}=\mathrm{Gal}(\overline{K}/K_{A,s})$ on a $\sigma_{|K_{A,s}}=\mathrm{id}$ et en particulier $\sigma_{|K_A}=\mathrm{id}$. La description de $I_{B}$ assure par ailleurs que $\tilde{c}(\sigma)=c(\sigma_{|L})=1$. Or $\sigma_{|L}\in \mathrm{Gal}(L/K_A)$ par ce qui précède et l'hypothèse sur $c$ donne $\sigma_{|L}=\mathrm{id}$. Il suit $\sigma_{|LK_{A,s}}=\mathrm{id}$ comme annoncé.

On a 
$$[K_{B,s}:K]= \Card \mathrm{Gal}(LK_{A,s}/K)/(I_{B}/\mathrm{Gal}(\overline{K}/LK_{A,s}))$$
et on déduit de l'étude précédente
$$\Card I_{B}/\mathrm{Gal}(\overline{K}/LK_{A,s})\mid \Card \mC.$$
Finalement on obtient
$$[L:K_A][K_{A,s}:K] \mid   [K_{B,s}:K]\Card \mC$$
comme annoncé. 
\end{proof}

\subsection{Les résultats d'existence}

Le théorème \ref{chap2main} est obtenu à partir des constructions suivantes. La première donne pour chaque premier impair l'existence de variétés abéliennes avec grosse monodromie finie sauvage.
\begin{prop} \label{chap2existpfix} Soient $K$ un corps de nombres et $g$ un entier non nul. On note $p$ un diviseur premier impair de $M(2g).$ On suppose qu'il existe une variété abélienne $A$ sur $K$ vérifiant
\begin{itemize}
\item[$(i)$] $\dim A= \frac{p-1}{2};$
\item[$(ii)$] $\End A\simeq \Z[\zeta_p];$
\item[$(iii)$] $A$ admet une polarisation principale;
\item[$(iv)$] $A$ a bonne réduction.
\end{itemize}
Soit $S$ un ensemble fini de places ultramétriques de $K$ contenant une place $v$ au-dessus de $p$. Alors il existe une extension finie $L$ de $K$ telle que,

\begin{itemize}
\item[(1)] au-dessus de chaque place de $S$ il existe une unique place de $L$ et elle est non ramifiée;
\item[(2)] si $w$ est la place de $L$ au-dessus de $v$ il existe une variété abélienne $B$ principalement polarisée de dimension $g$ sur $L$ vérifiant
$$\Card \Phi_{B,w}=p^{r(2g,p)}.$$
\end{itemize}
\end{prop}
\begin{proof}

On note $G$ le groupe $\Z/p\Z \wr \mathfrak{S}_n$ avec $n=\lfloor\frac{2g}{p-1} \rfloor$. La proposition \ref{grun} assure l'existence d'une extension $L$ de $K$ vérifiant la propriété (1) et qui admet de plus une extension galoisienne $M/L$ de groupe $G$ dont le groupe d'inertie au-dessus de $w$ est un $p$-Sylow de $G$.

Par les hypothèses sur $A$ on a $\operatorname{Aut} A^n \simeq  \mathrm{GL}_n (\Z[\zeta_p])$. Le groupe $G$ s'identifie alors à un sous-groupe unitaire de $\operatorname{Aut} (A^n)$ en considérant les matrices ayant exactement un élément de $\{1,\zeta_p,\dots, \zeta_p^{p-1}\}$ par ligne et par colonne (voir la proposition 4 de \cite{gl} pour plus de détails). On obtient une injection 
 $$c\colon \mathrm{Gal}(M/L) \longrightarrow \operatorname{Aut} (A_{M}^n).$$
 Le théorème \ref{twist} donne alors l'existence d'une variété abélienne $B'$, $M/L$-forme de $A^n_L$ avec pour groupe de monodromie finie en $w$ un $p$-Sylow de $G$ qui d'après le lemme \ref{cardcour} est de cardinal $p^{r(2g,p)}$. Soit une telle variété abélienne $B'$. 
 On peut choisir une polarisation principale de $A^n$ par produit d'une polarisation principale sur $A$ qui donne la conjugaison complexe comme involution de Rosati sur $\End A$. La polarisation produit est principale et son involution de Rosati est la composition de la transposition et de la conjugaison complexe sur les coefficients des matrices. Comme le cocycle que l'on considère a pour image un groupe de matrices unitaires la condition du lemme \ref{chap2lempolarprincip} est vérifiée et $B'$ admet une polarisation principale.

  La dimension de $B'$ est $b=\dim A^n=\frac{p-1}{2}\cdot \lfloor \frac{2g}{p-1} \rfloor\leq g$. Soit $k=g-b$. On considère $B=B'\times C^k$ où $C$ est une courbe elliptique CM sur $L$ ayant bonne réduction en $v_1$. Alors $B$ est une variété abélienne de dimension $g$ qui admet une polarisation principale et avec $\Phi_{B,w}\simeq  \Phi_{B',w}$ par construction. Il suit que $L$ vérifie aussi (2).
\end{proof}

La situation pour $p=2$ est plus délicate. On va construire des formes tordues de puissances de la courbe elliptique $E\colon y^2=x^3-x$ sur $\Q$. L'application du théorème \ref{twist2} avec un corps $L$ bien choisi va permettre de gagner un facteur $2$ par rapport à la construction générale utilisée pour les premiers impairs. Un corps $L$ qui convient est bien sûr donné par le théorème \ref{extgrun2}.

\begin{theo} \label{chap2exist2} Soit $g$ un entier naturel non nul. Il existe une variété abélienne $A$ principalement polarisée de dimension $g$ sur un corps de nombres $K$ telle que $A_{\overline{K}}$ est CM et $\Card \Phi_{A,v}=2^{\alpha}$ avec $\alpha\geq r(2g,2)+1-g$ pour une place $v$ de $K$ de corps résiduel de caractéristique $2$.
\end{theo}
\begin{proof}
Soit $E$ la courbe elliptique donnée par l'équation $y^2=x^3-x$ sur $\Q$. Le corps de définition des endomorphismes de $E_{\overline{\Q}}$ est $\Q(\mathrm{i})$ et $E$ a bonne réduction potentielle en $2$ avec pour groupe de monodromie finie $\Phi_{E,2}=Q_8$ le groupe des quaternions d'ordre $8$, comme le montre le tableau p. 358 de \cite{kraus} (on a $\Delta=2^6$ et $c_4=2^4\cdot 3$). Sur $\Q_2^{\mathrm{nr}}$ on a une tour $\Q_2^{\mathrm{nr}} \subset \Q_2^{\mathrm{nr}}(\mathrm{i})\subset (\Q_2^{\mathrm{nr}})_{E,s}$ où $\mathrm{Gal}((\Q_2^{\mathrm{nr}})_{E,s}/\Q_2^{\mathrm{nr}}(\mathrm{i}))$ est isomorphe à $\Z/4\Z$ et est son propre centralisateur dans $\Phi_{E,2}$.
Soit $A=E^g$. Le théorème \ref{extgrun2} appliqué avec $M=(\Q_2^{\mathrm{nr}})_{A,s}$ fournit un corps de nombres $K$ non ramifié en $2$ et une extension galoisienne $L$ de $K$ de groupe $H$ telle que

\begin{itemize}
\item[(1)] $H_2$ est le groupe d'inertie d'une place $w$ au-dessus de $2;$
\item[(2)] $L_w^{G_2}=K_v(\mathrm{i})$ où $v=w_{|K}$;
\item[(3)] $L_w \Q_2^{\mathrm{nr}}\cap (\Q_2^{\mathrm{nr}})_{A,s}=L_w^{\mathrm{nr}}\cap(\Q_2^{\mathrm{nr}})_{A,s}=(\Q_2^{\mathrm{nr}})_{A}=\Q_2^{\mathrm{nr}}(\mathrm{i}).$
\end{itemize} 

 On note $\widetilde{G}$ le sous-groupe $\mathrm{Gal}(L/K(\mathrm{i}))$ de $\mathrm{Gal}(L/K)$. L'injection $L\hookrightarrow L_w$ donne une inclusion $G_2\subset \widetilde{G}$ et le fait que $(1,\gamma)\in H_2$ agit non trivialement sur $\mathrm{i}$. Comme $\widetilde{G}$ est d'indice $2$ il est distingué et la suite exacte
$$\begin{tikzcd}
1\arrow[r] & \widetilde{G} \arrow[r]& H \arrow[r] & \Z/2\Z \arrow[r] & 1
\end{tikzcd}$$
est scindée par le choix de l'élément $(1,\gamma)\in H$. On a donc une écriture de $H$ comme produit semi-direct de $\widetilde{G}$ et  $\{(1,\mathrm{id}),(1,\gamma)\}$ ce qui permet de définir un cocycle
$$c\colon \mathrm{Gal}(L/K) \longrightarrow \operatorname{Aut} A_L$$
comme dans la démonstration de la proposition 2.2 de \cite{r}. Par construction ce cocycle vérifie que la restriction $c_{|\widetilde{G}}$ est injective.
  
Soit $B$ la variété abélienne obtenue comme $L/K$-forme de $A$ associée au cocycle $c$. Alors $B_{\Q_2^{\mathrm{nr}}}$ est la $L_w^{\mathrm{nr}}/\Q_2^{\mathrm{nr}}$-forme de $A_{\Q_2^{\mathrm{nr}}}$ associée à ce même cocycle restreint à $H_2$. De plus, la restriction $c_{|G_2}$ est injective du fait que $c_{|\widetilde{G}}$ l'est et $G_2\subset \widetilde{G}$. Les hypothèses du théorème \ref{twist2} sont vérifiées et celui-ci montre que $B$ vérifie l'énoncé. En effet on obtient les divisibilités 
$$[L_w^{\mathrm{nr}}:\Q_2^{\mathrm{nr}}(\mathrm{i})][\Q_2^{\mathrm{nr}}(\mathrm{i}) : \Q_2^{\mathrm{nr}}][(\Q_2^{\mathrm{nr}})_{A,s} : \Q_2^{\mathrm{nr}}(\mathrm{i})]\mid [(\Q_2^{\mathrm{nr}})_{B,s}: \Q_2^{\mathrm{nr}}]\cdot \Card \mC$$
et
$$[(\Q_2^{\mathrm{nr}})_{B,s}: \Q_2^{\mathrm{nr}}] \mid [L_w^{\mathrm{nr}}(\Q_2^{\mathrm{nr}})_{A,s} :\Q_2^{\mathrm{nr}}]$$
où $\mC$ est le centralisateur de $\mathrm{Gal}((\Q_2^{\mathrm{nr}})_{A,s} /\Q_2^{\mathrm{nr}}(\mathrm{i}))$ dans $\mathrm{Gal}((\Q_2^{\mathrm{nr}})_{A,s} /\Q_2^{\mathrm{nr}}).$ Cela donne $\Card \Phi_{B,v}=[(\Q_2^{\mathrm{nr}})_{B,s}: \Q_2^{\mathrm{nr}}]=2^{\alpha}$ pour un entier $\alpha$. De plus on a par construction $[L_w^{\mathrm{nr}}:\Q_2^{\mathrm{nr}}(\mathrm{i})][\Q_2^{\mathrm{nr}}(\mathrm{i}) : \Q_2^{\mathrm{nr}}]=\Card H_2=2^{r(2g,2)-g+1}$ et $[(\Q_2^{\mathrm{nr}})_{A,s} : \Q_2^{\mathrm{nr}}(\mathrm{i})]=\Card \mC=4$. On en déduit l'inégalité
$$\alpha \geq r(2g,2)+1-g.$$

On montre que $B$ est principalement polarisée avec le lemme \ref{chap2lempolarprincip} comme dans la démonstration précédente.
\end{proof}

On est finalement en mesure de démontrer le théorème principal.
\begin{theo}\label{Chap2theoprincipaldem}
Soient $g$ un entier naturel non nul et $K$ un corps de nombres non ramifié en $2$. On note $p_1,\dots,p_n$ les diviseurs premiers impairs de $M(2g)$. Alors il existe une extension finie $L$ de $K$ telle que pour chaque $i\in \{1,\dots,n\}$ il existe une variété abélienne $A_i$ de dimension $g$ principalement polarisée sur $L$ et une place $v_i$ de $L$ avec
$$\Card \Phi_{A_i,v_i}= p_i^{r(2g,p_i)}$$
et il existe une variété abélienne principalement polarisée $A$ de dimension $g$ sur $L$ et une place $v$ de $L$ telles que 
$$\Card \Phi_{A,v}=2^{\alpha}$$
et $\alpha \geq r(2g,2)+1-g$.
\end{theo}
\begin{proof}
On rappelle (voir paragraphe \ref{chap2rappel}) que les groupes de monodromie finie sont invariants par extension non ramifiée. 

On construit maintenant le corps $L$ de l'énoncé par compositum. Pour un corps de nombres $L'$ on note $S_{L'}$ l'ensemble des places ultramétriques de $L'$ divisant $2$ ou l'un des $p_i$.

Par le théorème \ref{chap2exist2} il existe un corps $L_2$ avec une variété abélienne principalement polarisée $A$ de dimension $g$ sur $L_2$ et une place $v$ au-dessus de $2$ de $L_2$ telles que 
$$\Card \Phi_{A,v}=2^{\alpha}$$
où $\alpha\geq r(2g,2)+1-g$.

Par ailleurs la proposition \ref{chap2existcm} assure l'existence de corps $K_{p_i}$ pour tous les $p_i,$ ramifiés seulement en $p_i$ et des places de caractéristiques résiduelles $p> \max p_i,$ tels que toute extension $K'/K_{p_i}$ vérifie les hypothèses de la proposition \ref{chap2existpfix} pour $S_{K'}$. On considère alors l'extension $K'$ obtenue par compositum de $K$, $L_2$ et les $K_{p_i}$. Par construction $K'/L_2$ est non ramifiée au-dessus de $v$. La proposition \ref{chap2existpfix} alors appliquée avec $K'$ et chacun des $p_i$ donne des extensions $L_{p_i}/K'$ non ramifiées au-dessus de $S_{K'}$ et avec une place $v_i$ au-dessus de $p_i$ ainsi qu'une variété abélienne $A_i$ de dimension $g$ principalement polarisée sur $L_{p_i}$ vérifiant
$$\Card \Phi_{A_i,v_i}= p_i^{r(2g,p_i)}.$$

Le corps $L/K'$ obtenu par compositum des $L_p$ est non ramifié au-dessus de $S_{K'}$ par construction et donc convient.
\end{proof}

\begin{rem}
On aurait pu énoncer le théorème sur un corps de nombres $K$ seulement modérément ramifié au-dessus de $2$. De plus, si on ne demande que l'existence de variétés abéliennes avec monodromie finie sauvage maximale pour les premiers impairs le théorème vaut sur tout corps de nombres $K$.

\end{rem}

\section{Une majoration dans le cas CM}

 On considère dans cette partie un corps de nombres  $K$, une variété abélienne $A$ de dimension $g$ sur $K$ et une place finie $v$ de $K$. On note comme précédemment $K_A$ le corps de définition des endomorphismes de $A_{\overline{K}}$ et $A'$ la variété abélienne $A_{K_A}$. Le but est l'obtention du théorème \ref{chap2majcmgen} que l'on déduit du cas isotypique.
 
 On suppose donc dans un premier temps que $A_{\overline{K}}$ est isotypique et CM. En particulier, par le théorème 6.(a) de \cite{serretate}, la variété abélienne $A$ a bonne réduction potentielle et, par le lemme \ref{chap2corpsCM}, $A'$ est isogène à une puissance d'une variété abélienne $B$ simple et CM sur $K_A$. On note $Z$ le corps $\End B \otimes \Q$ qui n'est autre que le centre de $\End A'\otimes \Q$. On introduit pour cette partie les notations suivantes :
$$2d=[Z:\Q],~h=\frac{g}{d} \text{ et } n=v_2(2d).$$ 
Avec ces notations $A'$ est isogène à $B^h$ et la dimension de $\End A'\otimes \Q$ sur son centre $Z$ est $h^2$.
  On fixe de plus une place $w$ de $K_A$ au-dessus de $v$. Comme les groupes de monodromie finie sont invariants par isogénie et puissance on a $\Phi_{A',w}\simeq \Phi_{B,w}$.

\begin{prop}\label{cormajophi} On a la relation de divisibilité 
$$\Card \Phi_{A,v} \mid [K_A:K] \Card \mu_Z$$
où $\mu_Z$ est le groupe des racines de l'unité de $Z$. 
\end{prop}
\begin{proof}
On note $L_{A'}$ et $L_A$ les extensions de $K_{A,w}$ et $K_v$ respectivement telles que $\mathrm{Gal}(L_{A'}/K_{A,w})=\Phi_{A',w}$ et $\mathrm{Gal}(L_A/K_v)=\Phi_{A,v}$. 
On a le diagramme d'extensions locales suivant
$$\begin{tikzcd}
 & L_{A'} \\
 L_A \arrow[ur, dash]& K_{A,w} \arrow[u, dash, "\Phi_{A',w}"'] \\
 & K_v \arrow[ul, dash, "\Phi_{A,v}"] \arrow[u, dash]
\end{tikzcd}$$
Il suit la relation de divisibilité
$$\Card \Phi_{A,v} \mid [K_{A,w}:K_v] \Card \Phi_{A',w}.$$
Or par le théorème 6.(b) de \cite{serretate} $\Card \Phi_{A',w}=\Card \Phi_{B,w}$ divise $\Card\mu_Z$ et par ailleurs $[K_{A,w}:K_v]$ divise $[K_A:K]$ du fait que $K_A/K$ est une extension galoisienne.
\end{proof}

On est alors amené à majorer la $2$-partie de $[K_A:K]$ en fonction de $Z$ et $g$. Le théorème 1.2 de \cite{gk} donne une borne de divisibilité pour $[K_A:K]$ optimale mais indépendante de $Z$ ce qui n'est pas suffisant ici.

On introduit les notations
 $$F^{(p)}=F\cap \Q(\mu_{p^{\infty}}),~ m(F,p)=\inf \{m\geq 1 \mid F^{(p)}\subset \Q(\mu_{p^m}),~p^m \neq 2 \}$$
  et $t(F,p)=[\Q(\mu_{p^{m(F,p)}}):F^{(p)}]$ où $F$ est un corps de nombres, $\mu_{p^m}$ le groupe des racines $p^m$-èmes de l'unité et $\Q(\mu_{p^{\infty}})$ la réunion des $\Q(\mu_{p^m})$ pour $m\geq 1$.  La borne de Schur est alors donnée, pour un entier naturel non nul $s$ et un corps de nombres $F,$ par
$$S(s,F)=2^{s-\lfloor s/t(F,2)\rfloor} \prod\limits_{p ~\mathrm{premier}} p^{m(F,p)\lfloor s/t(F,p)\rfloor} \Big ( \Big\lfloor \frac{s}{t(F,p)} \Big\rfloor ! \Big)_p,$$
où $\lfloor x\rfloor$ est la partie entière de $x$ et $m_p=p^{v_p(m)}$. On a en particulier
$$v_2(S(s,\Q({{{\mathrm{i}}}})))=r(2s,2)-s.$$

 D'après la preuve de la proposition 3.6 de \cite{r} on a la divisibilité
 $$[K_A :K] \mid [Z:\Q]\Gamma_{Z}(h)$$
 en notant 
 $$\Gamma_{Z}(h)=\operatorname{ppcm} \{\Card G \mid G\subset \mathfrak{A}^{\times}/Z^{\times}, [\mathfrak{A}:Z]=h^2\}$$
 où $\mathfrak{A}$ parcourt les algèbres centrales simples sur $Z$ de dimension $h^2$ et $G$ les sous-groupes finis de $\mathfrak{A}^{\times}/Z^{\times}$. 
 Les théorèmes 4.1 et 4.2 du même article relient $\Gamma_Z(d)$ à la borne de Schur et donnent ici
\begin{equation} \label{chap2divrémond}
[K_A :K]\Card \mu_{Z} \mid [Z:\Q] S(h, Z).
\end{equation}

 On va majorer la valuation $2$-adique du produit $[Z:\Q] S(h, Z).$ Avec les notations simplifiées $m(Z,2)=m,$ $t(Z,2)=t$ on a, pour tout entier naturel non nul $s$,
 $$v_2(S(s,Z))=s+(m-1)\Big\lfloor\frac{s}{t}\Big\rfloor+\sum\limits_{i=1}^{\infty} \Big\lfloor \frac{s}{2^i t}\Big\rfloor.$$
On remarque que $v_2(S(s,Z))$ ne dépend que de $m$ et $t$, c'est-à-dire que $v_2(S(s,Z))=v_2(S(s,Z^{(2)}))$. 

De cette formule exacte on peut déduire le lemme suivant qui nous ramène à étudier le cas d'une valuation maximale lorsque le degré de l'extension $Z/\Q$ est fixé. 
\begin{lem}\label{chap2casmax} Soient $s,d$ des entiers naturels non nuls et $n=v_2(2d)$. Le maximum des $v_2(S(s, Z))$ lorsque $Z$ parcourt les corps de nombres de degré $2d$ est obtenu pour une extension de degré $\frac{2d}{2^n}$ de $Z^{(2)}=\Q(\mu_{2^{n+1}})$. De plus, ce maximum est atteint seulement si $Z^{(2)}=\Q(\mu_{2^{n+1}})$. 
\end{lem}
\begin{proof}

On sait par le lemme 4.4 de \cite{r} que $t=1$ ou $2$. Comme la formule est croissante avec $m$ pour $t$ fixé il suffit de comparer les valeurs pour $m$ maximal dans les deux cas. Pour $t=1$ la valeur maximale possible de $m$ est $n+1$ et elle est $n+2$ pour $t=2$. Il suffit donc de vérifier 
$$s+ns+\sum\limits_{i=1}^{\infty} \Big\lfloor \frac{s}{2^i}\Big\rfloor - s-(n+1)\Big\lfloor\frac{s}{2}\Big\rfloor-\sum\limits_{i=1}^{\infty} \Big\lfloor \frac{s}{2^{i+1}}\Big\rfloor > 0.$$
Or on a
$$\Big\lfloor\frac{s}{2}\Big\rfloor+ns > (n+1)\Big\lfloor\frac{s}{2}\Big\rfloor$$
pour $n\geq 1$. 
\end{proof}

Appliqué à notre situation ce lemme donne l'inégalité suivante
 \begin{align*}v_2([Z:\Q] S(h, Z))&\leq v_2\big( 2^n S( h, \Q(\mu_{2^{n+1}}))\big).\\
\end{align*}

Par ailleurs le calcul suivant permet de se ramener à $\Q({{\mathrm{i}}})$ :
 $$v_2(S(h, \Q(\mu_{2^{n+1}}))) = v_2(S(2^{n-1}h, \Q({\mathrm{i}})))-3\cdot 2^{n-1}h+(n+2)h.$$
En effet, on a
\begin{multline*}v_2(S(h, \Q(\mu_{2^{n+1}})))-v_2(S(2^{n-1}h, \Q({\mathrm{i}})))= h+nh+\sum\limits_{i\geq 1}\Big\lfloor \frac{h}{2^{i}}\Big\rfloor \\- 2^{n-1}h-2^{n-1}h-\sum\limits_{i\geq 1}\Big\lfloor \frac{2^{n-1}h}{2^{i}}\Big\rfloor
\end{multline*}
\begin{align*}
&=(n+1)h-2 \cdot 2^{n-1}h -h \sum\limits_{i=0}^{n-2} 2^i \\
&=(n+2)h-3\cdot 2^{n-1}h.
\end{align*}

On peut maintenant conclure dans le cas isotypique.

\begin{theo}\label{maj2cm} Soit $A$ une variété abélienne de dimension $g$ sur un corps de nombres $K$ telle que $A_{\overline{K}}$ est isotypique et CM. Alors on a la majoration
$$v_2(\Card \Phi_{A,v}) \leq r(2g,2)-g+1$$
pour toute place ultramétrique $v$ de $K$. De plus l'égalité ne peut intervenir que lorsque $A_{\overline{K}}$ est isogène à une puissance de la courbe elliptique $y^2=x^3-x.$
\end{theo}
\begin{proof} Par la proposition \ref{cormajophi} et la relation de divisibilité (\ref{chap2divrémond}) on a
\begin{align*}
v_2 (\Card \Phi_{A,v})&\leq v_2([Z:\Q] S(h, Z)). \\
\end{align*}
Le calcul qui suit le lemme \ref{chap2casmax} permet alors de majorer $v_2 (\Card \Phi_{A,v})$ par
$$ v_2(S(2^{n-1}h,\Q({{\mathrm{i}}})))-3\cdot 2^{n-1}h+(n+2)h+n.$$
Finalement avec $h\geq 1$ et la croissance stricte en $s$ de $v_2(S(s,\Q(\mathrm{i})))$ on obtient
$$v_2 (\Card \Phi_{A,v}) \leq v_2(S(g,\Q({{\mathrm{i}}})))-3\cdot 2^{n-1}+2n+2.$$

On raisonne maintenant suivant la valeur de $n$. Si $n\geq 2$ on a $-3\cdot 2^{n-1}+2n+2\leq 0$ et il vient

$$v_2 (\Card \Phi_{A,v}) \leq v_2(S(g,\Q({{\mathrm{i}}})))= r(2g,2)-g.$$

Dans le cas $n=1,$ on distingue alors deux situations. Tout d'abord si $Z^{(2)}\neq \Q({\mathrm{i}})$ l'inégalité du lemme \ref{chap2casmax} est stricte donc on obtient
$$v_2 (\Card \Phi_{A,v}) < v_2(S(h,\Q({{\mathrm{i}}})))+1$$
ce qui donne
$$v_2(\Card \Phi_{A,v})\leq r(2g,2)-g.$$

Il reste la situation où $n=1$ et $Z^{(2)}=\Q({{\mathrm{i}}})$. La majoration est alors
\begin{align*}
v_2 (\Card \Phi_{A,v}) &\leq v_2(S(h,\Q({{\mathrm{i}}})))+1.
\end{align*} 
Si $Z\neq Z^{(2)}$ la croissance stricte en $s$ de $v_2(S(s,\Q({{\mathrm{i}}})))$ donne
$$v_2 (\Card \Phi_{A,v}) \leq v_2(S(g,\Q({{\mathrm{i}}})))=r(2g,2)-g$$
et si $Z=Z^{(2)}$ alors $h=g$ et l'inégalité est bien
$$v_2 (\Card \Phi_{A,v}) \leq v_2(S(g,\Q({{\mathrm{i}}})))+1=r(2g,2)-g+1.$$
Dans ce cas $Z=\Q(\mathrm{i})$ impose que $B$ est une courbe elliptique avec multiplication complexe par $\Q(\mathrm{i})$. La théorie de la multiplication complexe donne alors que $B_{\overline{K}}$ est isogène à la courbe d'équation $y^2=x^3-x$. 
\end{proof}

On termine cette partie par le passage du cas isotypique au cas général qui nécessite un dernier lemme technique de majoration inspiré du lemme 4.6 de \cite{gk}.
\begin{lem} \label{chap2lemmaj}Soit $p$ un nombre premier.
Pour des entiers naturels non nuls $s$ et $g$ tels que $s$ divise $g,$ on a l'inégalité 
$$sr(\frac{2g}{s},p)+v_p(s!)\leq r(2g,p).$$
\end{lem}
\begin{proof}
 Soit $k$ le plus petit entier naturel tel que $\Big\lfloor \frac{2g}{sp^k (p-1)}\Big\rfloor=0$. On a
 $$r(2g,p)=\sum\limits_{i=0}^{\infty} \Big\lfloor \frac{2g}{p^i(p-1)}\Big\rfloor=\sum\limits_{i=0}^{k-1} \Big\lfloor \frac{2g}{p^i(p-1)}\Big\rfloor+\sum\limits_{i=k}^{\infty} \Big\lfloor \frac{2g}{p^i(p-1)}\Big\rfloor.$$
Alors pour $i<k,$ c'est-à-dire $\Big\lfloor \frac{2g}{sp^i (p-1)}\Big\rfloor \neq 0,$ on a
$$ s\cdot\Big\lfloor \frac{2g}{s p^i(p-1)}\Big\rfloor  \leq \Big\lfloor \frac{2g}{p^i(p-1)}\Big\rfloor$$
d'où
$$sr(\frac{2g}{s},p)+\sum\limits_{i=1}^{\infty} \Big\lfloor \frac{2g}{p^{k-1}p^i(p-1)}\Big\rfloor\leq r(2g,p).$$

Par ailleurs, par choix de $k,$ $\Big\lfloor \frac{2g}{sp^{k-1}(p-1)}\Big\rfloor \geq 1$ et il suit que, pour $i\geq 1$,
$$\Big\lfloor \frac{2g}{p^{k-1}p^i (p-1)}\Big\rfloor =\Big\lfloor \frac{2g}{sp^{k-1} (p-1)}\cdot \frac{s}{p^i}\Big\rfloor \geq \Big\lfloor \frac{2g}{sp^{k-1} (p-1)}\Big\rfloor\cdot \Big\lfloor \frac{s}{p^i}\Big\rfloor \geq \Big\lfloor \frac{s}{p^i}\Big\rfloor$$
ce qui donne
$$\sum\limits_{i=1}^{\infty} \Big\lfloor \frac{2g}{p^{k-1}p^i(p-1)}\Big\rfloor \geq v_p(s!).$$
\end{proof}

Une disjonction de cas permet finalement d'obtenir la majoration.
\begin{theo} \label{chap2majcmgen2}
Soit $A$ une variété de dimension $g$ sur un corps de nombres $K$ telle que $A_{\overline{K}}$ est CM. Alors on a 
$$v_2(\Card \Phi_{A,v}) \leq r(2g,2)+1-g$$
pour toute place ultramétrique $v$ de $K$. De plus, l'égalité ne peut intervenir que lorsque une composante isotypique de $A_{\overline{K}}$ est isogène à une puissance de la courbe elliptique $y^2=x^3-x$. 
\end{theo}
\begin{proof}
On raisonne par récurrence sur la dimension $g$ de $A$.

Le cas où $A_{\overline{K}}$ est isotypique est traité dans le théorème \ref{maj2cm}. On suppose donc que $A_{\overline{K}}$ n'est pas isotypique.  

Soient $C_1,\dots, C_n$ les composantes isotypiques de $A_{\overline{K}}$, avec $n\geq 2$ par hypothèse, qui sont permutées par l'action de $\mathrm{Gal}(\overline{K}/K)$.

 S'il existe deux parties $(C_i)_{i\in I},(C_j)_{j\in J}\subset \{C_1,\dots, C_n\}$ non vides et stables par cette action telles que $I\cup J=\{1,\dots, n\}$, alors il existe des sous-variétés abéliennes non nulles $C$ et $D$ de $A$ telles que $C_{\overline{K}}=\sum_{i\in I} C_i$, $D_{\overline{K}}=\sum_{j\in J} C_j$ et $A=C+D$. En particulier on a $\Hom_{\overline{K}} (C_{\overline{K}}, D_{\overline{K}})=0$ ce qui assure que seule l'une des variétés $C_{\overline{K}}$ et $D_{\overline{K}}$ peut avoir une composante isotypique isogène à une puissance de la courbe $E\colon y^2=x^3-x$.  Soient $L_A,L_C$ et $L_D$ les plus petites extensions galoisiennes de $K_v^{\mathrm{nr}}$ sur lesquelles $A,C$ et $D$ atteignent réduction semi-stable. Alors on a l'inclusion $L_A\subset L_CL_D$ et comme le groupe de Galois de $L_CL_D$ sur $K_v^{\mathrm{nr}}$ est un sous-groupe du produit $\Phi_{C,v}\times \Phi_{D,v}$ on a l'inégalité
$$v_2(\Card \Phi_{A,v})\leq v_2(\Card \Phi_{C,v})+v_2(\Card \Phi_{D,v}).$$
On note $g_1=\dim C$ et $g_2=\dim D$. Par l'hypothèse de récurrence on obtient
\begin{align*}
v_2(\Card \Phi_{C,v})+v_2(\Card \Phi_{D,v})& \leq r(2g_1,2)-g_1+r(2g_2,2)-g_2+1 \\
&\leq r(2g,2)-g+1
\end{align*}
ce qui conclut dans ce cas. De plus, l'une des variétés $C_{\overline{K}}$ ou $D_{\overline{K}}$ a une composante isotypique isogène à une puissance de la courbe $E\colon y^2=x^3-x$ si et seulement si c'est le cas pour $A_{\overline{K}}$. Il suit que si $A_{\overline{K}}$ n'a pas de telle composante isotypique l'inégalité est stricte.

Sinon l'action est transitive. Les variétés abéliennes $C_i$ sont conjuguées sous l'action du groupe de Galois et ont donc même dimension $g/n$. Dans ce cas les algèbres $\End C_i\otimes \Q$ sont toutes isomorphes. La théorie de la multiplication complexe assure alors que si l'une des $C_i$ est isogène à une puissance de la courbe $E\colon y^2=x^3-x$ les autres le sont aussi, ce qui est absurde. 

Soit $L/K$ donné par le noyau de l'action de $\mathrm{Gal}(\overline{K}/K)$ sur les $C_i$. On a $[L:K]| n!$ et il existe des sous-variétés $D_i$ de $A_L$ telles que $(D_i)_{\overline{K}}=C_i$ et $A_L$ est isogène au produit des $D_i$. Par ailleurs pour une place $w|v$ de $L$ on a l'inégalité
$$v_2(\Card \Phi_{A,v})\leq v_2(\Card \Phi_{A_L,w}) +v_2([L:K]).$$
Or, par le théorème \ref{maj2cm} on a
$$v_2(\Card \Phi_{D_i,w})\leq r(\frac{2g}{n},2)-\frac{g}{n}$$
pour tout $i\geq 1$.
Il suit par le lemme \ref{chap2lemmaj} et l'hypothèse de récurrence
\begin{align*}
v_2(\Card \Phi_{A_L,w})&\leq \sum\limits_{i=1}^n v_2(\Card \Phi_{D_i,w}) \\
&\leq \sum\limits_{i=1}^n (r(\frac{2g}{n},2)-\frac{g}{n}) \\
&\leq r(2g,2)-g-v_2(n!).
\end{align*}
Finalement on a obtenu
\begin{align*}
v_2(\Card \Phi_{A,v})&\leq r(2g,2)-g-v_2(n!)+v_2([L:K]) \\
& \leq r(2g,2)-g.
\end{align*}

\end{proof}

\end{document}